\documentclass{amsart}
\usepackage[english]{babel}
\usepackage{amssymb,enumerate,amsmath}
\numberwithin{equation}{section}
\newtheorem{theorem}{Theorem}[section]
\newtheorem{thmA}{Theorem}

\newtheorem{corollary}[theorem]{Corollary}
\newtheorem{lemma}[theorem]{Lemma}

\newtheorem{remark}{Remark}

\newtheorem{conjecture}[thmA]{Conjecture}

\newcommand{\mathd}{\mathrm{d}}
\newcommand{\mathe}{\mathrm{e}}

\newcommand{\tmop}[1]{\ensuremath{\operatorname{#1}}}

\newenvironment{enumerateroman}{\begin{enumerate}[\textup{(}i\textup{)}] }{\end{enumerate}}

\newcommand{\dt}{\ensuremath{\frac{\partial}{\partial t}}}
\newcommand{\normho}{\ensuremath{| \mathring{h} |}}

\begin{document}

\title[A New Version of Huisken's Convergence Theorem]
{A New Version of Huisken's Convergence Theorem for Mean
Curvature Flow in Spheres}

\author{Li Lei and Hongwei Xu}
\address{Center of Mathematical Sciences \\ Zhejiang University \\ Hangzhou 310027 \\ China\\}
\email{lei-li@zju.edu.cn; xuhw@cms.zju.edu.cn}
\date{}
\thanks{Research supported by the National Natural Science Foundation of China, Grant No. 11371315.}

\begin{abstract}
  We prove that if the initial hypersurface of the mean curvature flow in
  spheres satisfies a sharp pinching condition, then the solution of the flow
  converges to a round point or a totally geodesic sphere. Our result improves the
  famous convergence theorem due to Huisken {\cite{MR892052}}. Moreover, we
  prove a convergence theorem under the weakly pinching condition. In
  particular, we obtain a classification theorem for weakly pinched
  hypersurfaces. It should be emphasized that our pinching condition implies that
  the Ricci curvature of the initial hypersurface is positive, but does not
  imply positivity of the sectional curvature.
\end{abstract}

{\maketitle}

\section{Introduction}

Let $F_{0} :M^{n} \rightarrow N^{n+1}$ be an $n$-dimensional immersed
hypersurface in a Riemannian manifold. The mean curvature flow with initial
value $F_{0}$ is a smooth family of immersions $F: M \times [ 0,T )
\rightarrow N^{n+1}$ satisfying
\begin{equation}
  \left\{\begin{array}{l}
    \dt F ( x,t ) = \vec{H} ( x,t ) ,\\
    F ( \cdot ,0 ) =F_{0} ,
  \end{array}\right.
\end{equation}
where $\vec{H} ( x,t )$ is the mean curvature vector of the hypersurface
$M_{t} =F_{t} ( M )$, $F_{t} =F ( \cdot ,t )$.

In 1984, Huisken {\cite{MR772132}} first proved that compact uniformly convex hypersurfaces
in Euclidean space will converge to a round point along the mean curvature
flow. Afterwards, Huisken{\cite{MR892052}} verified the following important
convergence result for mean curvature flow of pinched hypersurfaces in spheres.

\begin{thmA}Let $F_{0} :M^{n} \rightarrow \mathbb{F}^{n+1} (c)$ be an
n-dimensional ($n \geqslant 3$) closed hypersurface immersed in a
spherical space form of sectional curvature $c$. If $F_{0}$
satisfies $| h |^{2} < \frac{1}{n-1} H^{2} +2c$, then the mean
curvature flow with initial value $F_{0}$ converges to a round point
in finite time, or converges to a totally geodesic hypersurface as
$t \rightarrow \infty$.\end{thmA}

For any fixed positive constant $\varepsilon$, Huisken
{\cite{MR892052}} constructed examples to show that the pinching
condition above can not be improved to $| h |^{2} < \frac{1}{n-1}
H^{2} +2c+ \varepsilon$. An attractive problem is: Is it possible to
improve Huisken's pinching condition? During the past nearly three
decades, there has been no progress on this aspect. For more
convergence results of mean curvature flow, we refer the readers to
{\cite{MR2739807,baker2011mean,MR837523,lx2013,lx2014,liu2011extension,MR3078951,liu2012mean}}.

Let $M$ be an $n$-dimensional hypersurface in a space form $\mathbb{F}^{n+1}
(c)$ with constant curvature $c$. Set
\begin{equation}
  \alpha ( n,H,c ) := n c+ \frac{n}{2 ( n-1 )} H^{2}
  - \frac{n-2}{2 ( n-1 )} \sqrt{H^{4} +4 ( n-1 ) c H^{2}} .
\end{equation}

Based on the pioneering work on closed minimal submanifolds in a
sphere due to Simons {\cite{MR0233295}}, Lawson \cite{MR0238229} and
Chern-do Carmo-Kobayashi {\cite{MR0273546}} proved the famous
rigidity theorem for $n$-dimensional closed minimal submanifolds in
$\mathbb{S}^{n+q}$ whose squared norm of the second fundamental form
satisfies $| h |^{2}\leqslant n/(2-1/q)$. After the work due to Okumura
\cite{MR0317246,MR0353216} and Yau \cite{Yau}, the second author
\cite{Xu1} verified the generalized Simons-Lawson-Chern-do
Carmo-Kobayashi theorem for compact submanifolds with parallel mean
curvature in $\mathbb{S}^{n+q}$. In particular, Cheng-Nakagawa
{\cite{MR1050421}} and Xu {\cite{Xu1}} got the following rigidity
theorem for constant mean curvature hypersurfaces, independently.

\begin{thmA}Let $M^{n}$ be a compact hypersurface with constant mean
curvature in $\mathbb{S}^{n+1} \left( 1/ \sqrt{c} \right)$. If $| h
|^{2} \leqslant \alpha ( n,H,c )$, then either $M$ is the totally
umbilical sphere $\mathbb{S}^{n} \left( n/ \sqrt{H^{2} +n^{2} c}
\right)$, one of the Clifford torus $\mathbb{S}^{k} \left(
\sqrt{k/(n c)} \right) \times
    \mathbb{S}^{n-k} \left( \sqrt{(n-k)/(n c)} \right)$, $1\leqslant k\leqslant n-1$, or
the isoparametric hypersurface $\mathbb{S}^{n-1} \left( 1/ \sqrt{c+
\lambda^{2}} \right) \times \mathbb{S}^{1} \left( \lambda /
\sqrt{c^{2} +c \lambda^{2}} \right)$, where
$\lambda = \frac{| H | + \sqrt{H^{2} +4 ( n-1 )c}}{2 ( n-1 )}$.
\end{thmA}

For the refined rigidity results in higher codimensions, we refer
the readers to {\cite{MR1161925,MR1241055}}. Motivated by Theorem B,
the optimal topological sphere theorem due to Shiohama-Xu
\cite{MR1458750} and Andrews' suggestion on the nonlinear parabolic
flow \cite{andrews2002mean}, we have the following conjecture (see
\cite{liu2011extension}).

\begin{conjecture}
  Let $F_{0} :M^{n} \rightarrow \mathbb{S}^{n+1} \left( 1/ \sqrt{c} \right)$
  be an $n$-dimensional closed hypersurface satisfying $| h |^{2} < \alpha (
  n,H,c )$. Then the mean curvature flow with initial
  value $F_{0}$ converges to a round point in finite time or converges to a
  totally geodesic sphere as $t \rightarrow \infty$.
\end{conjecture}

In particular, noting that $\min_{H}   \alpha ( n,H,1 ) =2
\sqrt{n-1}$, we have the following.

\begin{conjecture}
  Let $F_{0} :M^{n} \rightarrow \mathbb{S}^{n+1} ( 1 )$ be an $n$-dimensional
  closed hypersurface satisfying $| h |^{2} <2 \sqrt{n-1}$. Then
  the mean curvature flow with initial value $F_{0}$ converges to a
  round point or a totally geodesic sphere.
\end{conjecture}

In {\cite{lxz2013optimal}}, Li, Xu and Zhao investigated the
conjectures above, and proved the following convergence result for the mean
curvature flow in a sphere.

\begin{thmA}Let $F_{0} :M^{n} \rightarrow \mathbb{S}^{n+1} ( 1 )$ be an
n-dimensional ($n \geqslant 3$) closed hypersurface immersed in the unit
sphere. If there exists a positive constant $\varepsilon < \frac{1}{8}$, such
that $F_{0}$ satisfies
\[ | h |^{2} \leqslant n-4 \varepsilon + \frac{n}{2 ( n-1+ \varepsilon )}
   H^{2} - \frac{n-2}{2 ( n-1+ \varepsilon )} \sqrt{H^{4} +4 ( n-1 ) H^{2}}
   \hspace{1em} \tmop{and} \hspace{1em} H \geqslant
   \frac{n}{\sqrt[4]{\varepsilon}} , \]
then the mean curvature flow with initial value $F_{0}$ converges to a round
point in finite time.\end{thmA}

The purpose of the present paper is to prove a sharp convergence
theorem for the mean curvature flow of hypersurfaces in spherical space forms, which is a
refined version of the famous convergence theorem due to Huisken
{\cite{MR892052}}.

\begin{theorem}
  \label{theo1}Let $F_{0} :M^{n} \rightarrow \mathbb{F}^{n+1} (
  c)$ be an n-dimensional ($n \geqslant 3$) closed hypersurface
  immersed in a spherical space form. If $F_{0}$ satisfies
  \[ | h |^{2} < \gamma ( n,H,c ) , \]
  then the mean curvature flow with initial value $F_{0}$ has a unique smooth
  solution $F: M \times [ 0,T ) \rightarrow \mathbb{F}^{n+1} (c)$,
    and $F_{t}$ converges to a round point in finite time, or
  converges to a totally geodesic hypersurface as $t \rightarrow \infty$.
    Here $\gamma ( n,H,c )$ is an explicit
  positive scalar defined by
\[ \gamma ( n,H,c ) = \min \{ \alpha ( H^{2} ) , \beta ( H^{2} ) \} , \]
where
\[ \alpha ( x ) =n c+ \frac{n}{2 ( n-1 )} x- \frac{n-2}{2 ( n-1 )} \sqrt{x^{2}
   +4 ( n-1 ) c x} , \]
\[ \beta ( x ) = \alpha ( x_{0} ) + \alpha' ( x_{0} ) ( x-x_{0} ) +
   \frac{1}{2} \alpha'' ( x_{0} ) ( x-x_{0} )^{2} , \]
\[ x_{0} =y_{n} c, \hspace{1em} y_{n} =4 ( 1-n ) + \frac{2 ( n^{2} -4
   )}{\sqrt{2n-5}}   \cos \left( \frac{1}{3} \arctan \tfrac{n^{2} -4n+6}{2 (
   n-1 ) \sqrt{2n-5}} \right) . \]
\end{theorem}

\begin{remark}
  Notice that $\gamma ( n,H,c ) > \frac{1}{n-1} H^{2} +2c$ and $\sqrt{8} n^{2}
  >y_{n}$. Furthermore, a computation shows that $\gamma ( n,H,c ) >
\frac{9}{5} \sqrt{n-1} c$. Therefore, Theorem \ref{theo1}
substantially improves Theorem A as well as Theorem E.
\end{remark}

Put $\alpha_{1} ( x ) =n+ \frac{n}{2 ( n-1 )} x- \frac{n-2}{2 ( n-1 )} \sqrt{x^{2}
   +4 ( n-1 ) x}$.
As a consequence of Theorem \ref{theo1}, we obtain the following
convergence result.

\begin{theorem}\label{theo3}
  Let $F_{0} :M^{n} \rightarrow \mathbb{F}^{n+1} ( c )$ be an $n$-dimensional
  ($n \geqslant 3$) closed hypersurface immersed in a spherical space form. If $F_{0}$
  satisfies
    \[  | h |^{2} < k_{n}c, \]
  then the mean curvature flow with
  initial value $F_{0}$ has a unique smooth solution $F: M \times [ 0,T )
  \rightarrow \mathbb{F}^{n+1} ( c )$, and $F_{t}$ converges to a round point
  in finite time, or converges to a totally geodesic hypersurface as $t \rightarrow
  \infty$. Here $k_n$ is an explicit
  positive constant defined by \[ k_n=
     \left\{\begin{array}{ll}
     \alpha_1 ( y_n ) - \alpha_1' ( y_n ) y_n + \frac{1}{2}
  \alpha_1'' ( y_n) y_n^{2} , & n=3 ,\\
     \alpha_{1} ( y_{n} ) - \frac{\alpha_{1}' (
      y_{n} )^{2}}{2 \alpha_{1}'' ( y_{n} )} , & n \geqslant 4.
   \end{array}\right.\]
\end{theorem}

\begin{remark} By a computation, we have
$k_n > \frac{9}{5} \sqrt{n-1}$.
  In particular, if $5 \leqslant n \leqslant 10$, then $k_n>1.999 \sqrt{n-1}$. In fact, $k_{10}=6$.
  This shows that the pinching constant $k_{n}c$ in Theorem \ref{theo3} is
  sharp.
\end{remark}

\begin{corollary}\label{n195}
  Let $F_{0} :M^{n} \rightarrow \mathbb{F}^{n+1} ( c )$ be an $n$-dimensional
  ($n \geqslant 3$) closed hypersurface immersed in a spherical space form.
  If $F_{0}$
  satisfies
  \[ | h |^{2} \leqslant \frac{9}{5} \sqrt{n-1} c, \]
  then the mean curvature flow with
  initial value $F_{0}$ has a unique smooth solution $F: M \times [ 0,T )
  \rightarrow \mathbb{F}^{n+1} ( c )$, and $F_{t}$ converges to a round point
  in finite time, or converges to a totally geodesic hypersurface as $t \rightarrow
  \infty$.
\end{corollary}

If the ambient space is a sphere, we have the sharp differentiable
sphere theorem.

\begin{theorem}
  \label{theo5}
  Let $M_0$ be an n-dimensional ($n \geqslant 3$) closed hypersurface immersed
  in $\mathbb{S}^{n+1} \left( 1/ \sqrt{c} \right)$ which
	satisfies $| h |^{2} < \gamma ( n,H,c )$, then $M_0$ is diffeomorphic
  to the standard $n$-sphere $\mathbb{S}^{n}$. In particular, if $M_{0}$
  satisfies
  $| h |^{2} < k_{n}c, $
  then $M_0$ is diffeomorphic
  to $\mathbb{S}^{n}$.
\end{theorem}

Furthermore, for compact hypersurfaces in spheres, we have the following
convergence theorem for the mean curvature flow under the weakly pinching
condition.

\begin{theorem}
  \label{theo2}Let $F_{0} :M^{n} \rightarrow \mathbb{S}^{n+1} \left( 1/ \sqrt{c} \right)$ be
an n-dimensional ($n \geqslant 3$) closed hypersurface immersed in a sphere.
If $F_{0}$ satisfies
\[ | h |^{2} \leqslant \gamma ( n,H,c ) , \]
then the mean curvature flow with initial value $F_{0}$ has a unique smooth
solution $F: M \times [ 0,T ) \rightarrow \mathbb{S}^{n+1} \left( 1/ \sqrt{c}
\right)$, and either
  \begin{enumerateroman}
    \item $T$ is finite, and $F_{t}$ converges to a round point as $t
    \rightarrow T$,

    \item $T= \infty$, and $F_{t}$ converges to a totally geodesic sphere as
    $t \rightarrow \infty$, or

    \item $T$ is finite, $M_{t}$ is congruent to $\mathbb{S}^{n-1} (
    r_{1} ( t ) ) \times \mathbb{S}^{1} ( r_{2} ( t ) )$, where $r_{1} ( t
    )^{2} +r_{2} ( t )^{2} =1/c$, $r_{1} ( t )^{2} = \frac{n-1}{n c} ( 1-
    \mathe^{2n c ( t-T )} )$, and $F_{t}$ converges to a great circle as
    $t \rightarrow T$.
  \end{enumerateroman}
\end{theorem}

Theorem \ref{theo2} (iii) shows that our pinching conditions in
Theorem \ref{theo2} is sharp. As a consequence of Theorem
\ref{theo2}, we have the following classification theorem.

\begin{corollary}
  Let $M_0$ be an n-dimensional ($n \geqslant 3$) closed hypersurface immersed
  in $\mathbb{S}^{n+1} \left( 1/ \sqrt{c} \right)$ which satisfies $| h |^{2}
  \leqslant \gamma ( n,H,c )$. Then $M_0$ is either diffeomorphic to the
  standard $n$-sphere $\mathbb{S}^{n}$, or congruent to $\mathbb{S}^{n-1} (
  r_{1} ) \times \mathbb{S}^{1} ( r_{2} )$, where $r_{1}^{2} +r_{2}^{2} =1/c$
  and $r_{1}^{2} < \frac{n-1}{n c}$.
\end{corollary}

From the fact that $k_{10}=6$, we get the following optimal
convergence result.
\begin{corollary}
  \label{theo4}Let $F_{0} :M \rightarrow \mathbb{S}^{11} \left( 1/ \sqrt{c} \right)$ be
a $10$-dimensional closed hypersurface which satisfies $  | h |^{2}
\leqslant 6c$. Then the mean curvature flow with initial value
$F_{0}$ has a unique smooth solution $F: M \times [ 0,T )
\rightarrow \mathbb{S}^{11} \left( 1/ \sqrt{c} \right)$, and either
  \begin{enumerateroman}
    \item $T$ is finite, and $F_{t}$ converges to a round point as $t
    \rightarrow T$,

    \item $T= \infty$, and $F_{t}$ converges to a totally geodesic sphere as
    $t \rightarrow \infty$, or

    \item $T$ is finite, $M_{t}$ is congruent to $\mathbb{S}^{9} (
    r_{1} ( t ) ) \times \mathbb{S}^{1} ( r_{2} ( t ) )$, where $r_{1} ( t
    )^{2} +r_{2} ( t )^{2} =1/c$, $r_{1} ( t )^{2} = \frac{9}{10 c} ( 1-\frac{1}{6}
    \mathe^{20 c  t } )$, and $F_{t}$ converges to a great circle as
    $t \rightarrow T$.
  \end{enumerateroman}
In particular, $M_0$ is diffeomorphic to $\mathbb{S}^{10}$, or
congruent to $\mathbb{S}^{9} \left(
  \sqrt{\frac{3}{4c}} \right) \times \mathbb{S}^{1} \left( \sqrt{\frac{1}{4c}} \right)$.
\end{corollary}

\section{Preservation of curvature pinching}

Let $F: M^{n} \times [ 0,T ) \rightarrow \mathbb{F}^{n+1} (c)$ be a
mean curvature flow in a spherical space form. Let $( x^{1} , \cdots
,x^{n} )$ be the local coordinates of an open neighborhood in $M$.
We consider the hypersurface $M_{t}$ at time $t$. In the local
coordinates, the first fundamental form of $M_{t}$ can be written as
symmetric matrix $( g_{i j} )$. Denote by $( g^{i j} )$ the inverse
matrix of $( g_{i j} )$. With the suitable choice of the unit normal
vector field, the second fundamental form of $M_{t}$ can be written
as symmetric matrix $( h_{i j} )$. We adopt the Einstein summation
convention. Let $h_{i}^{j} =g^{j k} h_{i k}$ and $h^{i j} =g^{i k}
g^{j l} h_{k l}$. Then the mean curvature and the squared norm of
the second fundamental form of $M_{t}$ can be written as
$H=h_{i}^{i}$ and $| h |^{2} =h^{i j} h_{i j}$, respectively. Let
$\mathring{h}_{i j} =h_{i j} - \tfrac{H}{n}  g_{i j}$ be the
traceless second fundamental form, whose squared norm satisfies
$\normho^{2} = | h |^{2} - \frac{H^{2}}{n} $.

Similar to Lemma 1.2 of {\cite{MR892052}}, the gradient and Laplacian of the
second fundamental form have the following properties.

\begin{lemma}
  \label{dA2}For any hypersurface of $\mathbb{F}^{n+1} ( c )$, we have
  \begin{enumerateroman}
    \item $| \nabla h |^{2} \geqslant \frac{3}{n+2} | \nabla H |^{2}$,

    \item $\Delta | h |^{2} =2 \langle h, \nabla^{2} H \rangle +2 | \nabla h
    |^{2} +2 W$, where $W=H h_{i}^{j}  h_{j}^{k}  h_{k}^{i} - | h |^{4} +n c
    \normho^{2}$.
  \end{enumerateroman}
\end{lemma}

The evolution equations take the same form as in {\cite{MR892052}}.

\begin{lemma}
  \label{evo}For the mean curvature flow $F: M \times [ 0,T ) \rightarrow
  \mathbb{F}^{n+1} ( c )$, we have
  \begin{enumerateroman}
    \item $\dt g_{i j} =-2 H h_{i j}$,

    \item $\dt h_{i j} = \Delta h_{i j} -2H h_{i k} h^{k}_{j} + | h |^{2} h_{i
    j} +2c H g_{i j} -n c h_{i j}$,

    \item $\dt H= \Delta H+H ( | h |^{2} +n c )$,

    \item $\dt | h |^{2} = \Delta | h |^{2} -2 | \nabla h |^{2} +4 c H^{2} +2
    | h |^{4} -2n c | h |^{2}$,

    \item $\dt \normho^{2} = \Delta \normho^{2} -2 | \nabla h |^{2} +
    \frac{2}{n} | \nabla H |^{2} +2 \normho^{2} ( | h |^{2} -n c )$.
  \end{enumerateroman}
\end{lemma}

We define
\[ \alpha ( x ) =n c+ \frac{n}{2 ( n-1 )} x- \frac{n-2}{2 ( n-1 )} \sqrt{x^{2}
   +4 ( n-1 ) c x} , \hspace{1em} x \geqslant 0, \]
and
\[ \beta ( x ) = \alpha ( x_{0} ) + \alpha' ( x_{0} ) ( x-x_{0} ) +
   \frac{1}{2} \alpha'' ( x_{0} ) ( x-x_{0} )^{2} , \hspace{1em} x \geqslant
   0, \]
where
\[ x_{0} =y_{n} c, \hspace{1em} y_{n} =4 ( 1-n ) + \frac{2 ( n^{2} -4
   )}{\sqrt{2n-5}}   \cos \left( \frac{1}{3} \arctan \tfrac{n^{2} -4n+6}{2 (
   n-1 ) \sqrt{2n-5}} \right) . \]
Then we set
\begin{equation}\label{defgama}
 \gamma ( x ) = \left\{\begin{array}{ll}
     \alpha ( x ) , & x \geqslant x_{0} ,\\
     \beta ( x ) , & 0 \leqslant x<x_{0} .
   \end{array}\right.
\end{equation}

It is easy to see that $\gamma$ is a $C^{2}$-function on $[ 0,+ \infty )$.
Moreover, we obtain the following.

\begin{lemma}
  $\label{app}$For $n \geqslant 3$, $c>0$ and $x \geqslant 0$, the
  $C^{2}$-function $\gamma ( x )$ satisfies
  \begin{enumerateroman}
    \item $2x \gamma'' ( x ) + \gamma' ( x ) \leqslant \frac{3}{n+2}$,
    and the equality holds if and only if $x=x_{0}$,

    \item $( \gamma ( x ) +n c ) x  \gamma' ( x ) \geqslant 2 c x+  \gamma ( x
    )^{2} -n c  \gamma ( x )$, and the equality holds if and only if $x
    \geqslant x_{0}$,

    \item $\gamma ( x ) >x \gamma' ( x )$,

    \item $\gamma ( x ) = \min \{ \alpha ( x ) , \beta ( x ) \}$,

    \item $\frac{x}{n-1} +2c< \gamma ( x ) < \frac{x}{n-1} +n c$,

    \item $\frac{n-2}{\sqrt{n ( n-1 )}} \sqrt{x \left( \gamma ( x ) -
    \frac{1}{n} x \right)} + \gamma ( x ) \leqslant \frac{2}{n} x+n c$.
  \end{enumerateroman}
\end{lemma}

\begin{proof}
  (i) By direct computations, for $x>0$, we get
  \[ \alpha' ( x ) = \frac{n}{2 ( n-1 )} - \frac{n-2}{2 ( n-1 )}   \frac{x+2 (
     n-1 ) c}{\sqrt{x^{2} +4 ( n-1 ) c x}} , \]
  \[ \alpha'' ( x ) = \frac{2 ( n-2 ) ( n-1 ) c^{2}}{( x^{2} +4 ( n-1 ) c x
     )^{3/2}} , \]
  \[ \alpha''' ( x ) =- \frac{6 ( n-2 ) ( n-1 ) ( x+2 ( n-1 ) c ) c^{2}}{(
     x^{2} +4 ( n-1 ) c x )^{5/2}} . \]
  Notice that $y_{n}$ is the only positive root of equation
  \begin{equation}
    y_{n} ( y_{n} +6 ( n-1 ) )^{2} = \left( \frac{n^{2} -4n+6}{n^{2} -4}
    \right)^{2} ( y_{n} +4 ( n-1 ) )^{3} . \label{bneq}
  \end{equation}

  Let $\varphi_{1} ( x ) =2x \alpha'' ( x ) + \alpha' ( x ) = \frac{n}{2 ( n-1
  )} - \frac{( n-2 ) x^{2} ( x+6 ( n-1 ) c )}{2 ( n-1 ) ( x^{2} +4 ( n-1 ) c x
  )^{3/2}}$. From (\ref{bneq}) we get $\varphi_{1} ( x_{0} ) = \frac{3}{n+2}$.
  We have $\lim_{x \rightarrow \infty} \varphi_{1} ( x ) = \frac{1}{n-1}$.
  Since $\varphi_{1}' ( x ) =- \frac{6 ( n-2 ) ( n-1 ) c^{2} x^{2}}{( x^{2} +4
  ( n-1 ) c x )^{5/2}} <0$, we get $\frac{1}{n-1} < \varphi_{1} ( x ) <
  \frac{3}{n+2}$ if $x>x_{0}$.

  Let $\varphi_{2} ( x ) =2x \beta'' ( x ) + \beta' ( x ) = \alpha' ( x_{0} )
  + \alpha'' ( x_{0} ) ( 3x-x_{0} )$. Then we get $\varphi_{2}' ( x ) =3
  \alpha'' ( x_{0} ) >0$. Thus we have $\varphi_{2} ( x ) < \varphi_{2} (
  x_{0} ) = \frac{3}{n+2}$ if $0 \leqslant x<x_{0}$. Hence assertion (i) follows.

  (ii) It's easy to check that $\alpha$ and $\alpha'$ satisfy
  \begin{equation}
    ( \alpha ( x ) +n c ) x  \alpha' ( x ) =2c x+ \alpha ( x )^{2} -n c
    \alpha ( x ) . \label{alid1}
  \end{equation}

  Let $\psi ( x ) = ( \beta ( x ) +n c ) x  \beta' ( x ) -2c x- \beta ( x
  )^{2} +n c  \beta ( x )$. From the $C^{2}$-continuity of $\gamma$ and
  (\ref{alid1}), we get $\psi ( x_{0} ) = \psi' ( x_{0} ) =0$. Making
  calculation, we get
  \[ \psi'' ( x ) =3 \alpha'' ( x_{0} ) [ \alpha'' ( x_{0} ) x^{2} + ( \alpha'
     ( x_{0} ) -x_{0} \alpha'' ( x_{0} ) ) x+n c ] . \]

  Let $x_{1} = \left( n \sqrt{n-1} -2n+2 \right) c$. We have $\alpha' ( x_{1}
  ) =0$, $\alpha'' ( x_{1} ) = \frac{2}{( n-2 )^{2} \sqrt{n-1} c}$ and
  $\varphi_{1} ( x_{1} ) = \frac{4}{2 \sqrt{n-1} +n}$. Since $\varphi_{1}' ( x
  ) <0$ and $\varphi_{1} ( x_{0} ) \leqslant \varphi_{1} ( x_{1} )$, we get
  $x_{0} \geqslant x_{1}$. Since $\alpha'' ( x ) >0$, we have $\alpha' ( x_{0} )
  \geqslant \alpha' ( x_{1} ) =0$. By the definition of $y_{n}$, we have
  $y_{n} <4 ( 1-n ) + \frac{2 ( n^{2} -4 )}{\sqrt{2n-5}} \leqslant
  \frac{2}{15} n ( n+2 ) = \frac{2n}{5 \varphi_{1} ( x_{0} )}$. This yields
  \[ 2x_{0} \alpha'' ( x_{0} ) + \alpha' ( x_{0} ) < \frac{2n c}{5x_{0}} .\] Thus
  we have $\psi'' ( x ) \geqslant 3 \alpha'' ( x_{0} ) n c \left( 1-
  \frac{x}{5x_{0}} \right)$. If $0 \leqslant x<x_{0}$, then $\psi'' ( x ) >0$.
  Therefore, we have $\psi' ( x ) <0< \psi ( x )$ if $0<x < x_{0}$.
  This proves (ii).

  (iii) We have
  \[ \alpha ( x ) -x  \alpha' ( x ) =n c- \frac{( n-2 ) c x}{\sqrt{x^{2} +4 (
     n-1 ) c x}} >0. \]
  Since $( \beta ( x ) -x \beta' ( x ) )' =- \alpha'' ( x_{0} ) x \leqslant
  0$, we get $\beta ( x ) -x \beta' ( x ) \geqslant \alpha ( x_{0} ) -x_{0}
  \alpha' ( x_{0} ) >0$ if $0 \leqslant x<x_{0}$. Hence assertion (iii) is
  proved.

  (iv) We have $\alpha''' ( x ) <0= \beta''' ( x )$. From $\alpha ( x_{0}
  ) = \beta ( x_{0} )$, $\alpha' ( x_{0} ) = \beta' ( x_{0} )$ and $\alpha'' (
  x_{0} ) = \beta'' ( H_{0} )$, we obtain $\alpha ( x ) > \beta ( x )$ for $0
  \leqslant x<x_{0}$, and $\alpha ( x ) < \beta ( x )$ for $x>x_{0}$.

  (v) It's easy to verify that $2c< \alpha ( x ) - \frac{x}{n-1} \leqslant n
  c$. Since $\alpha'' ( x ) >0$ and $\beta'' ( x ) = \alpha'' ( x_{0} ) >0$,
  we get $\gamma'' ( x ) >0$. So, $\lim_{x \rightarrow \infty}   \gamma' ( x )
  = \frac{1}{n-1}$ implies $\gamma' ( x ) < \frac{1}{n-1}$. Then we have
  $\gamma ( x ) - \frac{x}{n-1} > \lim_{x \rightarrow \infty} (\alpha ( x ) -
  \frac{x}{n-1}) =2c$. This proves (v).

  (vi) Note that $\alpha$ satisfies the following identity
  \begin{equation}
    \frac{n-2}{\sqrt{n ( n-1 )}} \sqrt{x  \left( \alpha ( x ) - \frac{1}{n} x
    \right)} + \alpha ( x ) = \frac{2}{n} x+n c. \label{alid2}
  \end{equation}
  Combing (\ref{alid2}) and $\gamma ( x ) \leqslant \alpha ( x )$, we prove
  (vi).
\end{proof}

The following lemma will be used in the proofs of Corollaries
\ref{n195} and \ref{theo4}.

\begin{lemma}
  We have $\gamma ( x ) >  \frac{9}{5} \sqrt{n-1}  c$, where $\gamma(x)$ is
    defined by (\ref{defgama}). In particular, $k_n>1.999\sqrt{n-1}$ for $5 \leqslant n \leqslant 9$, and $k_{10}=6$.
\end{lemma}

\begin{proof}
  Let $x_{1} = \left( n \sqrt{n-1} -2n+2 \right) c$. Since $\alpha' ( x_{1} )
  =0$ and $a'' ( x ) >0$, we get $\alpha ( x ) \geqslant \alpha ( x_{1} ) =2
  \sqrt{n-1} c$.

  If $n=3$, we have $\beta ( x ) >0.027 \frac{x^{2}}{c} +0.304x+2.661c>1.8
  \sqrt{2} c$.

    If $n\geqslant 4$, by the definition of $\beta$, we have $\beta ( x ) \geqslant
  \alpha ( x_{0} ) - \frac{\alpha' ( x_{0} )^{2}}{2 \alpha'' ( x_{0} )} =k_{n}
  c$.

  Making a calculation, we get $k_{4} >3.443>1.8 \sqrt{3}$ and $k_{5}
  >3.998=1.999 \times 2$.

  If $n \geqslant 6$, putting $x_{2} = \sqrt{2 ( n-1 )} \left( \sqrt{n-1} -1/
  \sqrt{2} \right)^{2} c$, we have $x_{0} < [ 4 ( 1-n ) + \frac{2 ( n^{2}
  -4 )}{\sqrt{2n-5}} ] c<x_{2}$. Then we get $\alpha' ( x_{0} ) <
  \alpha' ( x_{2} ) = \frac{1}{2n-3}$ and $\alpha'' ( x_{0} ) > \alpha'' (
  x_{2} ) = \frac{4 \sqrt{2} ( n-2 )}{\sqrt{n-1} ( 2n-3 )^{3} c}$. Thus we
  obtain $\beta ( x ) >2 \sqrt{n-1} c- \frac{\sqrt{n-1} ( 2n-3 ) c}{8 \sqrt{2}
  ( n-2 )} > \frac{9}{5} \sqrt{n-1} c$.

    In fact, by more computations, we get $k_n>1.999\sqrt{n-1}$ if $5 \leqslant n \leqslant 9$, and $k_{10}=6$.
\end{proof}

Let $\omega \in C^{2} [ 0,+ \infty )$ be a positive function which takes the
following form
\[ \omega ( x ) = \frac{x^{2}}{\sqrt{x^{2} +4 ( n-1 ) c x}} \left[ \left( 1+
   \tfrac{n^{2} c}{x} \right) \tfrac{\frac{n}{n-2} - ( 1+4 ( n-1 ) c /x
   )^{-1/2}}{\frac{n}{n-2} +  ( 1+4 ( n-1 ) c /x )^{-1/2}} \right]^{2}
   \hspace{1em} \tmop{for} \hspace{0.5em} x \geqslant x_{0} . \]
Then $\omega$ has the following properties.

\begin{lemma}
  \label{wpp}For $n \geqslant 3$, $c>0$ and $x \geqslant 0$, $\omega$
  satisfies
  \begin{enumerateroman}
    \item $( \gamma ( x ) +n c ) x  \frac{\omega' ( x )}{\omega ( x )} =2
    \gamma ( x ) -x  \gamma' ( x ) -3 n c$ if $x \geqslant x_{0}$,

    \item $2x_{0}   \omega'' ( x_{0} ) + \omega' ( x_{0} ) >0$, $\lim_{x
    \rightarrow \infty}  (2x  \omega'' ( x ) + \omega' ( x ) ) >0$,

    \item $\omega ( x ) -x  \omega' ( x )$ is bounded.
  \end{enumerateroman}
\end{lemma}

\begin{proof}
  (i) When $x \geqslant x_{0}$, we see that $\omega$ satisfies the
  following identity
  \begin{equation}
    \frac{\omega' ( x )}{\omega ( x )} = \frac{2  \alpha ( x ) -x  \alpha' ( x
    ) -3 n c}{x ( \alpha ( x ) +n c )} , \label{dlnw}
  \end{equation}
  which implies (i).

  (ii) By a computation, we get that $\lim_{x \rightarrow \infty} ( 2x  \omega'' ( x ) +
  \omega' ( x ) )= \frac{1}{( n-1 )^{2}}$.

  If $n=3$, we have $2x_{0}   \omega'' ( x_{0} ) + \omega' ( x_{0} ) \approx
  11.4>0$.

  If $n \geqslant 4$, we have $y_{n} <4 ( 1-n ) + \frac{2 ( n^{2} -4
  )}{\sqrt{2n-5}} < ( n-2 )^{2}$. So, we get $\alpha ( x_{0} ) < \alpha ( ( n-2
  )^{2} c ) =n c$. From (\ref{dlnw}), for $x \geqslant x_{0}$, we have
  \begin{eqnarray*}
    \frac{2x \omega'' ( x ) + \omega' ( x )}{\omega ( x )} & \geqslant & 2x
    \left( \frac{\omega' ( x )}{\omega ( x )} \right)' + \frac{\omega' ( x
    )}{\omega ( x )}\\
    & = & \frac{2 \alpha' ( x ) ( x \alpha' ( x ) +5n c )}{( \alpha ( x ) +n
    c )^{2}} + \frac{5n c-2x^{2} \alpha'' ( x ) -x  \alpha' ( x )}{x ( \alpha
    ( x ) +n c )} - \frac{2}{x} .
  \end{eqnarray*}
  Combing the above inequality with $\alpha ( x_{0} ) <n c$, $\alpha' ( x_{0}
  ) \geqslant 0$ and $2x_{0} \alpha'' ( x_{0} ) + \alpha' ( x_{0} ) < \frac{2n
  c}{5x_{0}}$, we obtain $2x_{0}   \omega'' ( x_{0} ) + \omega' ( x_{0} ) >0$.

  (iii) The assertion follows from $\lim_{x \rightarrow \infty} (\omega ( x )
  -x  \omega' ( x ) )= \frac{2 ( 2n-1 ) c}{n-1}$.
\end{proof}

For convenience, we denote $\gamma ( H^{2} )$, $\gamma' ( H^{2} )$,
$\gamma'' ( H^{2} )$, $\omega (H^{2} )$, $\omega' ( H^{2} )$ and $\omega'' ( H^{2} )$
by $\gamma$, $\gamma'$, $\gamma''$, $\omega$, $\omega'$
and $\omega''$, respectively.

Suppose that $M_{0}$ is a compact hypersurface satisfying $| h |^{2} < \gamma$.
Then there exists a sufficiently small positive number $\varepsilon$, such
that
\[ | h |^{2} < \gamma - \varepsilon   \omega . \]
Now we show that this pinching condition is preserved.

\begin{theorem}
  \label{pinch}If $M_{0}$ satisfies $| h |^{2} < \gamma - \varepsilon
  \omega$, then this condition holds for all time $t \in [ 0,T )$.
\end{theorem}

\begin{proof}
  Let $U= | h |^{2} - \gamma + \varepsilon   \omega$. From the evolution
  equations we have
  \begin{eqnarray*}
    \left( \dt - \Delta \right) U & = & -2 | \nabla h |^{2} +2 ( 2H^{2}
    \gamma'' + \gamma' - \varepsilon ( 2H^{2}   \omega'' + \omega' ) ) |
    \nabla H |^{2}\\
    &  & +4 c H^{2} +2 | h |^{4} -2n c | h |^{2} -2 ( \gamma' - \varepsilon
    \omega' ) \cdot H^{2} ( | h |^{2} +n c ) .
  \end{eqnarray*}
  By Lemma \ref{app} (i) and Lemma \ref{wpp} (ii), the coefficient of $|
  \nabla H |^{2}$ in the above formula is less than $\frac{6}{n+2}$. Then
  Lemma \ref{dA2} (i) yields $-2 | \nabla h |^{2} + \frac{6}{n+2} | \nabla H
  |^{2} \leqslant 0$. Then replacing $| h |^{2}$ by $U+ \gamma - \varepsilon
  \omega$, we get
  \begin{eqnarray}
    \left( \dt - \Delta \right) U & \leqslant & 2U^{2} +2U [ 2  \gamma -H^{2}
    \gamma' -n c+ \varepsilon ( H^{2}   \omega' -2  \omega ) ] \nonumber\\
    &  & +4 c H^{2} +2  \gamma^{2} -2 n c  \gamma -2 ( \gamma +n c ) H^{2}
    \gamma' \\
    &  & +2 \varepsilon \omega \left[ - 2 \gamma +H^{2}   \gamma' + n c+ (
    \gamma +n c ) H^{2} \frac{  \omega'}{\omega} \right] +2 \varepsilon^{2}
    \omega \cdot ( \omega -H^{2}   \omega' ) . \nonumber
  \end{eqnarray}
  By Lemma \ref{app} (ii), Lemma \ref{wpp} (i) and (iii), when $\varepsilon$ is
  small enough, we have
  \[ \left( \dt - \Delta \right) U<2U^{2} + \upsilon \cdot U. \]
  Therefore, the conclusion follows from the maximum principle.
\end{proof}

\section{An estimate for traceless second fundamental form}

In this section, we derive an estimate for the traceless second fundamental
form, which shows that the principal curvatures will approach each other along
the mean curvature flow.

\begin{theorem}
  \label{sa0h2}If $M_{0}$ satisfies $| h |^{2} < \gamma -
  \varepsilon   \omega$, then there exist constants $0< \sigma <1$
  and $C_{0} >0$ depending only on $M_{0}$, such that for all $t \in [ 0,T )$,
  we have
  \[ \normho^{2} \leqslant C_{0} ( H^{2} +c )^{1- \sigma} \mathe^{-2 \sigma c
     t} . \]
\end{theorem}

Let $\mathring{\gamma} = \gamma - \frac{1}{n} H^{2}$. Theorem \ref{pinch}
says that $\normho^{2} < \mathring{\gamma} - \varepsilon   \omega$ holds for
all time. We denote by $\mathring{\gamma}' = \gamma' - \frac{1}{n}$,
$\mathring{\gamma}'' = \gamma''$ the first and second derivatives of
$\mathring{\gamma}$ with respect to $H^{2}$. For $0< \sigma <1$, we set
\[ f_{\sigma} = \frac{\normho^{2}}{\mathring{\gamma}^{1- \sigma}} . \]

To prove Theorem \ref{sa0h2}, we need to show that $f_{\sigma}$ decays
exponentially. First, we make an estimate for the time derivative of $f_{\sigma}$.

\begin{lemma}
  \label{ptf}There exists a positive constants $C_{1}$ depending only on $n$,
  $c$, such that \
  \[ \dt f_{\sigma} \leqslant \Delta f_{\sigma} + \frac{2C_{1}}{\normho} |
     \nabla f_{\sigma} | | \nabla H | - \frac{2 \varepsilon  f_{\sigma}}{C_{1}
     \normho^{2}} | \nabla H |^{2} +2 \sigma ( | h |^{2} -n c ) f_{\sigma} .
  \]
\end{lemma}

\begin{proof}
  By a straightforward calculation, we have
  \[ \dt f_{\sigma} =f_{\sigma} \left( \frac{\dt \normho^{2}}{\normho^{2}} - (
     1- \sigma ) \frac{\dt \mathring{\gamma}}{\mathring{\gamma}} \right) . \]
  The gradient of $f_{\sigma}$ can be written as
  \[ \nabla f_{\sigma} =f_{\sigma} \left( \frac{\nabla
     \normho^{2}}{\normho^{2}} - ( 1- \sigma ) \frac{\nabla
     \mathring{\gamma}}{\mathring{\gamma}} \right) . \]
  The Laplacian of $f_{\sigma}$ is given by
  \begin{equation}
    \Delta f_{\sigma} =f_{\sigma} \left( \frac{\Delta
    \normho^{2}}{\normho^{2}} - ( 1- \sigma ) \frac{\Delta
    \mathring{\gamma}}{\mathring{\gamma}} \right) -2 ( 1- \sigma )
    \frac{\left\langle \nabla f_{\sigma} , \nabla \mathring{\gamma}
    \right\rangle}{\mathring{\gamma}} + \sigma ( 1- \sigma ) f_{\sigma}
    \frac{\left| \nabla \mathring{\gamma} \right|^{2}}{\left|
    \mathring{\gamma} \right|^{2}} . \label{lapf}
  \end{equation}
  By the evolution equations, we have
  \begin{eqnarray}
    \left( \dt - \Delta \right) f_{\sigma} & = & 2 ( 1- \sigma )
    \frac{\left\langle \nabla f_{\sigma} , \nabla \mathring{\gamma}
    \right\rangle}{\mathring{\gamma}} - \sigma ( 1- \sigma ) f_{\sigma}
    \frac{\left| \nabla \mathring{\gamma} \right|^{2}}{\left|
    \mathring{\gamma} \right|^{2}} \nonumber\\
    &  & + \frac{2 f_{\sigma}}{\normho^{2}} \left( \frac{| \nabla H |^{2}}{n}
    - | \nabla h |^{2} \right) +2 ( 1- \sigma ) f_{\sigma} \frac{2H^{2}
    \mathring{\gamma}'' + \mathring{\gamma}'}{\mathring{\gamma}}   | \nabla H
    |^{2} \nonumber\\
    &  & +2f_{\sigma} \left[ ( | h |^{2} -n c ) - ( 1- \sigma ) ( | h |^{2}
    +n c ) \frac{H^{2}   \mathring{\gamma}'}{\mathring{\gamma}} \right]
    \nonumber\\
    & \leqslant & \frac{2}{\mathring{\gamma}} | \nabla f_{\sigma} | \left|
    \nabla \mathring{\gamma} \right|  \label{dtf}\\
    &  & +2f_{\sigma} \left[   \frac{1}{\normho^{2}} \left( \frac{| \nabla H
    |^{2}}{n} - | \nabla h |^{2} \right) + ( 1- \sigma ) \frac{2H^{2}
    \mathring{\gamma}'' + \mathring{\gamma}'}{\mathring{\gamma}} | \nabla H
    |^{2} \right] \nonumber\\
    &  & +2f_{\sigma} \left[ \sigma ( | h |^{2} +n c ) + ( 1- \sigma ) ( | h
    |^{2} +n c ) \left( 1- \frac{H^{2}
    \mathring{\gamma}'}{\mathring{\gamma}} \right) -2 n c \right]  . \nonumber
  \end{eqnarray}

  From the definitions of $\mathring{\gamma}$, $\omega$, there exist a
  positive constant $C_{1}$ depending only on $n$, $c$, such that
  \[ \frac{2 | H | \left| \mathring{\gamma}'
     \right|}{\sqrt{\mathring{\gamma}}} <C_{1} \hspace{1em} \tmop{and}
     \hspace{1em} \frac{n ( n+2 )   \mathring{\gamma}}{2 ( n-1 ) \omega}
     <C_{1} . \]
  This together with $\normho^{2} < \mathring{\gamma}$ implies
  \begin{equation}
    \frac{\left| \nabla \mathring{\gamma} \right|}{\mathring{\gamma}} =
    \frac{2 \left| \mathring{\gamma}' \right| | H | | \nabla H
    |}{\mathring{\gamma}} \leqslant C_{1} \frac{| \nabla H |}{\normho} .
    \label{aopao}
  \end{equation}
  Next we estimate the expression in the first square bracket of the right
  hand side of (\ref{dtf}). Lemma \ref{app} (i) implies $2H^{2}
  \mathring{\gamma}'' + \mathring{\gamma}' \leqslant \frac{2 ( n-1 )}{n ( n+2
  )}$. From Lemma \ref{dA2} (i), we have
  \begin{eqnarray*}
    &  & \frac{1}{\normho^{2}} \left( \frac{| \nabla H |^{2}}{n} - | \nabla h
    |^{2} \right) + ( 1- \sigma ) \frac{2H^{2} \mathring{\gamma}'' +
    \mathring{\gamma}'}{\mathring{\gamma}} | \nabla H |^{2}\\
    & \leqslant & \left( \frac{2 ( 1-n )}{n ( n+2 )}   \frac{1}{\normho^{2}}
    + \frac{( 1- \sigma )}{\mathring{\gamma}}   \frac{2 ( n-1 )}{n ( n+2 )}
    \right) | \nabla H |^{2}\\
    & \leqslant & \frac{2 ( n-1 )}{n ( n+2 )} \left(
    \frac{1}{\mathring{\gamma}} - \frac{1}{\normho^{2}} \right) | \nabla H
    |^{2}\\
    & \leqslant & - \frac{2 ( n-1 )}{n ( n+2 )}   \frac{\varepsilon
    \omega}{\mathring{\gamma} \normho^{2}} | \nabla H |^{2}\\
    & \leqslant & - \frac{\varepsilon}{C_{1} \normho^{2}} | \nabla H |^{2} .
  \end{eqnarray*}

  Then we estimate the expression in the second square bracket of the right
  hand side of (\ref{dtf}). From Lemma \ref{app} (iii), we get $( | h |^{2} +n
  c ) \left( 1- \frac{H^{2}   \mathring{\gamma}'}{\mathring{\gamma}} \right)
  \leqslant \frac{1}{\mathring{\gamma}} ( \gamma +n c ) ( \gamma -H^{2}
  \gamma' )$. Lemma \ref{app} (ii) yields $( \gamma +n c ) ( \gamma
  -H^{2} \gamma' ) \leqslant 2n c  \mathring{\gamma}$. Thus we obtain
  \[ \sigma ( | h |^{2} +n c ) + ( 1- \sigma ) ( | h |^{2} +n c ) \left( 1-
     \frac{H^{2}   \mathring{\gamma}'}{\mathring{\gamma}} \right) -2 n c
     \leqslant \sigma ( | h |^{2} -n c ) . \]

  This completes the proof of the Lemma \ref{ptf}.
\end{proof}

To estimate the term $2 \sigma | h |^{2} f_{\sigma}$ in Lemma \ref{dtf}, we
need the following.

\begin{lemma}
  \label{lapa0}If $M$ is an n-dimensional $(n \geqslant 3 )$ hypersurface in
  $\mathbb{F}^{n+1} (c)$ $(c>0)$ which satisfies $\normho^{2} <
  \mathring{\gamma} - \varepsilon   \omega$, then there exists a positive
  constant $C_{2}$ depending only on $n$ and $c$, such that
  \[ \Delta \normho^{2} \geqslant 2 \left\langle \mathring{h} , \nabla^{2} H
     \right\rangle +2 \varepsilon  C_{2} | h |^{2} \normho^{2} . \]
\end{lemma}

\begin{proof}
  From Lemma \ref{dA2}, we have
  \[ \Delta \normho^{2} =2  \left\langle \mathring{h} , \nabla^{2} H
     \right\rangle +2 | \nabla h |^{2} - \frac{2}{n} | \nabla H |^{2} +2 W
     \geqslant 2  \left\langle \mathring{h} , \nabla^{2} H \right\rangle +2W.
  \]
  Let $\lambda_{i}$ $(1 \leqslant i \leqslant n ) $ be the principal curvatures of $M$. Put
  $\mathring{\lambda}_{i} = \lambda_{i} - \frac{H}{n}$. Then $\sum
  \mathring{\lambda}_{i} =0$, $\sum \mathring{\lambda}_{i}^{2} = \normho^{2}$.
  Thus we have
  \[ h_{i}^{j}  h_{j}^{k}  h_{k}^{i} = \sum_{i} \lambda_{i}^{3} = \sum_{i}
     \mathring{\lambda}_{i}^{3} + \frac{3}{n} H  \normho^{2} + \frac{1}{n^{2}}
     H^{3} . \]
  Using the inequality $\left| \sum_{i} \mathring{\lambda}_{i}^{3} \right|
  \leqslant \frac{n-2}{\sqrt{n ( n-1 )}} \normho^{3}$ (see {\cite{MR0353216}},
  Lemma 2.1), we have
  \begin{eqnarray}
    W & = & H h_{i}^{j}  h_{j}^{k}  h_{k}^{i} - | h |^{4} +n c \normho^{2}
    \nonumber\\
    & \geqslant & -  \frac{n-2}{\sqrt{n ( n-1 )}} | H | \normho^{3} +
    \frac{1}{n} H^{2} \normho^{2} - \normho^{4} +n c \normho^{2} \\
    & \geqslant & \normho^{2} \left( -  \frac{n-2}{\sqrt{n ( n-1 )}} | H |
    \sqrt{\mathring{\gamma}} + \frac{1}{n} H^{2} - \left( \mathring{\gamma} -
    \varepsilon   \omega \right) +n c \right) . \nonumber
  \end{eqnarray}
  Let $C_{2} = \inf ( \omega / \gamma )$. It follows from Lemma \ref{app} (vi)
  that $W \geqslant \varepsilon   \omega \normho^{2} \geqslant \varepsilon
  C_{2} | h |^{2} \normho^{2}$.
\end{proof}

From (\ref{lapf}), (\ref{aopao}) and Lemma \ref{lapa0}, we have
\begin{eqnarray*}
  \Delta f_{\sigma} & \geqslant & f_{\sigma} \frac{\Delta
  \normho^{2}}{\normho^{2}} - ( 1- \sigma ) f_{\sigma} \frac{\Delta
  \mathring{\gamma}}{\mathring{\gamma}} -2 ( 1- \sigma ) \frac{\left\langle
  \nabla f_{\sigma} , \nabla \mathring{\gamma}
  \right\rangle}{\mathring{\gamma}}\\
  & \geqslant & \frac{2 f_{\sigma}}{\normho^{2}}   \left\langle \mathring{h}
  , \nabla^{2} H \right\rangle +2 \varepsilon  C_{2} | h |^{2} f_{\sigma} - (
  1- \sigma ) \frac{f_{\sigma}}{\mathring{\gamma}}   \Delta \mathring{\gamma}
  - \frac{2C_{1}}{\normho} | \nabla f_{\sigma} | | \nabla H | .
\end{eqnarray*}
This is equivalent to
\begin{equation}
  2 \varepsilon  C_{2} | h |^{2} f_{\sigma} \leqslant \Delta f_{\sigma} -
  \frac{2f_{\sigma}}{\normho^{2}}   \left\langle \mathring{h} , \nabla^{2} H
  \right\rangle + \frac{2C_{1}}{\normho} | \nabla f_{\sigma} | | \nabla H | +
  ( 1- \sigma ) \frac{f_{\sigma}}{\mathring{\gamma}}   \Delta
  \mathring{\gamma} . \label{2eC2}
\end{equation}
Multiplying both sides of the above inequality by $f_{\sigma}^{p-1}$, then
integrating them over $M_{t}$, applying the divergence theorem and the
relation $\nabla_{i} \mathring{h}_{j}^{i} = \frac{n-1}{n}   \nabla_{j} H$, we
get
\begin{equation}
  \int_{M_{t}} f_{\sigma}^{p-1} \Delta f_{\sigma} \mathd \mu_{t} =- ( p-1 )
  \int_{M_{t}} f_{\sigma}^{p-2} | \nabla f_{\sigma} |^{2} \mathd \mu_{t}
  \leqslant 0,
\end{equation}
\begin{eqnarray}
  - \int_{M_{t}} \frac{f_{\sigma}^{p}}{\normho^{2}}   \left\langle
  \mathring{h} , \nabla^{2} H \right\rangle \mathd \mu_{t} & = & -
  \int_{M_{t}} \frac{f_{\sigma}^{p-1}}{\mathring{\gamma}^{1- \sigma}}
  \mathring{h}^{i j} \nabla^{2}_{i,j} H  \mathd \mu_{t} \nonumber\\
  & = & \int_{M_{t}} \nabla_{i} \left(
  \frac{f_{\sigma}^{p-1}}{\mathring{\gamma}^{1- \sigma}} \mathring{h}^{i j}
  \right) \nabla_{j} H  \mathd \mu_{t} \nonumber\\
  & = & \int_{M_{t}} \left[ ( p-1 )
  \frac{f_{\sigma}^{p-2}}{\mathring{\gamma}^{1- \sigma}} \mathring{h}^{i j}
  \nabla_{i} f_{\sigma}   \nabla_{j} H \right. \nonumber\\
  &  & \left. - ( 1- \sigma ) \frac{f_{\sigma}^{p-1}}{\mathring{\gamma}^{2-
  \sigma}}   \mathring{h}^{i j}   \nabla_{i} \mathring{\gamma}   \nabla_{j} H+
  \frac{n-1}{n}   \frac{f_{\sigma}^{p-1}}{\mathring{\gamma}^{1- \sigma}} |
  \nabla H |^{2} \right] \mathd \mu_{t} \\
  & \leqslant & \int_{M_{t}} \left[ ( p-1 )
  \frac{f_{\sigma}^{p-2}}{\mathring{\gamma}^{1- \sigma}} \normho | \nabla
  f_{\sigma} | | \nabla H | \right. \nonumber\\
  &  & \left. + \frac{f_{\sigma}^{p-1}}{\mathring{\gamma}^{2- \sigma}}
  \normho \left| \nabla \mathring{\gamma} \right| |   \nabla H | +
  \frac{f_{\sigma}^{p-1}}{\mathring{\gamma}^{1- \sigma}} | \nabla H |^{2}
  \right] \mathd \mu_{t} \nonumber\\
  & \leqslant & \int_{M_{t}} \left[ ( p-1 ) \frac{f_{\sigma}^{p-1}}{\normho}
  | \nabla f_{\sigma} | |   \nabla H | + ( C_{1} +1 )
  \frac{f_{\sigma}^{p}}{\normho^{2}}   |   \nabla H |^{2} \right] \mathd
  \mu_{t} , \nonumber
\end{eqnarray}
and
\begin{eqnarray}
  \int_{M_{t}} \frac{f_{\sigma}^{p}}{\mathring{\gamma}}   \Delta
  \mathring{\gamma} \mathd \mu_{t} & = & - \int_{M_{t}} \left\langle \nabla
  \left( \frac{f_{\sigma}^{p}}{\mathring{\gamma}} \right) , \nabla
  \mathring{\gamma} \right\rangle \mathd \mu_{t} \nonumber\\
  & = & \int_{M_{t}} \left[ -p  \frac{f_{\sigma}^{p-1}}{\mathring{\gamma}}
  \left\langle \nabla f_{\sigma} , \nabla \mathring{\gamma} \right\rangle
  +f_{\sigma}^{p} \left( \frac{\left| \nabla \mathring{\gamma}
  \right|}{\mathring{\gamma}} \right)^{2} \right] \mathd \mu_{t}
  \label{intfglap}\\
  & \leqslant & \int_{M_{t}} \left[ p C_{1}
  \frac{f_{\sigma}^{p-1}}{\normho}   | \nabla f_{\sigma} | |   \nabla H |
  +C_{1}^{2} \frac{f_{\sigma}^{p}}{\normho^{2}} | \nabla H |^{2} \right]
  \mathd \mu_{t} . \nonumber
\end{eqnarray}
Putting (\ref{2eC2})-(\ref{intfglap}) together, we obtain
\begin{equation}
  2 \varepsilon   \int_{M_{t}} | h |^{2} f_{\sigma}^{p} \mathd \mu_{t}
  \leqslant C_{3} \int_{M_{t}} \left[ \frac{p f_{\sigma}^{p-1}}{\normho}   |
  \nabla f_{\sigma} | |   \nabla H |  +  \frac{f_{\sigma}^{p}}{\normho^{2}}
  |   \nabla H |^{2} \right]   \mathd \mu_{t} , \label{nhofps}
\end{equation}
where $C_{3}$ is a positive constant depending on $n$ and $c$.

Using (\ref{nhofps}) and Lemma \ref{ptf}, we make an estimate for the time
derivative of the integral of $f_{\sigma}^{p}$.
\begin{eqnarray}
  \frac{\mathd}{\mathd t} \int_{M_{t}} f_{\sigma}^{p} \mathd \mu_{t} & = & p
  \int_{M_{t}} f_{\sigma}^{p-1} \frac{\partial f_{\sigma}}{\partial t} \mathd
  \mu_{t} - \int_{M_{t}} f_{\sigma}^{p} H^{2} \mathd \mu_{t} \nonumber\\
  & \leqslant & p \int_{M_{t}} \left[ f_{\sigma}^{p-1} \Delta f_{\sigma}  +
  \frac{2C_{1} f_{\sigma}^{p-1}}{\normho} | \nabla f_{\sigma} | | \nabla H |
  \right. \nonumber\\
  &  & \left. - \frac{2 \varepsilon  f_{\sigma}^{p}}{C_{1} \normho^{2}} |
  \nabla H |^{2} + 2 \sigma ( | h |^{2} -n c ) f_{\sigma}^{p} \right] \mathd
  \mu_{t} \nonumber\\
  & \leqslant & p \int_{M_{t}} \left[ - ( p-1 ) f_{\sigma}^{p-2} | \nabla
  f_{\sigma} |^{2} + \frac{2C_{1} f_{\sigma}^{p-1}}{\normho} | \nabla
  f_{\sigma} | | \nabla H | - \frac{2 \varepsilon  f_{\sigma}^{p}}{C_{1}
  \normho^{2}} | \nabla H |^{2} \right.  \label{dtint}\\
  &  & \left. + \frac{C_{3} \sigma  p}{\varepsilon}
  \frac{f_{\sigma}^{p-1}}{\normho} | \nabla f_{\sigma} | |   \nabla H | +
  \frac{C_{3}   \sigma}{\varepsilon}   \frac{f_{\sigma}^{p}}{\normho^{2}}   |
  \nabla H |^{2} -2 \sigma n c f_{\sigma}^{p} \right] \mathd \mu_{t}
  \nonumber\\
  & = & p \int_{M_{t}} f_{\sigma}^{p-2} \left[ - ( p-1 ) | \nabla f_{\sigma}
  |^{2} + \left( 2C_{1} + \frac{C_{3} \sigma  p}{\varepsilon} \right)
  \frac{f_{\sigma}}{\normho} | \nabla f_{\sigma} | | \nabla H | \right.
  \nonumber\\
  &  & \left. - \left( \frac{2 \varepsilon}{C_{1}} - \frac{C_{3}
  \sigma}{\varepsilon} \right) \frac{f_{\sigma}^{2}}{\normho^{2}} | \nabla H
  |^{2} \right] \mathd \mu_{t} -2p \sigma n c  \int_{M_{t}} f_{\sigma}^{p}
  \mathd \mu_{t} . \nonumber
\end{eqnarray}

In the following we show that the $L^{p}$-norm of $f_{\sigma}$ decays
exponentially.

\begin{lemma}
  \label{pnorm}There exists a constant $C_{4}$ depending on $M_{0}$ such that
  for all $p \geqslant 8C_{1}^{3} \varepsilon^{-1}$ and $\sigma \leqslant
  \varepsilon^{2} p^{-1/2}$, we have
  \[ \left( \int_{M_{t}} f_{\sigma}^{p} \mathd \mu_{t} \right)^{\frac{1}{p}}
     <C_{4}   \mathe^{-3 \sigma c t} . \]
\end{lemma}

\begin{proof}
  The expression in the square bracket of the right hand side of (\ref{dtint})
  is a quadratic polynomial. With $\varepsilon$ small enough, its discriminant
  satisfies
  \begin{eqnarray*}
    &  & \left( 2C_{1} + \frac{C_{3} \sigma  p}{\varepsilon} \right)^{2} -4 (
    p-1 ) \left( \frac{2 \varepsilon}{C_{1}} - \frac{C_{3}
    \sigma}{\varepsilon} \right)\\
    & < & 8C_{1}^{2} +2C_{3}^{2} \varepsilon^{2} p-2p  \varepsilon /C_{1}\\
    & = & ( 8C_{1}^{2} -p  \varepsilon /C_{1} ) + ( 2C_{3}^{2}
    \varepsilon^{2} p-p  \varepsilon /C_{1} )\\
    & < & 0.
  \end{eqnarray*}
  Therefore, this quadratic polynomial is non-positive. So we have
  \[ \frac{\mathd}{\mathd t} \int_{M_{t}} f_{\sigma}^{p} \mathd \mu_{t}
     \leqslant -3p \sigma c  \int_{M_{t}} f_{\sigma}^{p} \mathd \mu_{t} . \]
  This implies $\int_{M_{t}} f_{\sigma}^{p} \mathd \mu_{t} \leqslant
  \mathe^{-3p \sigma c t} \int_{M_{0}} f_{\sigma}^{p} \mathd \mu_{0}$.
\end{proof}

Letting $g_{\sigma} =f_{\sigma}   \mathe^{2 \sigma c t}$, we get

\begin{corollary}
  \label{pnorm2}There exist a constant $C_{5}$ depending on $M_{0}$ such that
  for all $r \geqslant 0$, $p \geqslant \max \{ 4 r^{2}   \varepsilon^{-4}
  ,8C_{1}^{3} \varepsilon^{-1} \}$ and $\sigma \leqslant   \frac{1}{2}
  \varepsilon^{2}  p^{-1/2}$, we have
  \[ \left( \int_{M_{t}} | h |^{2 r}  g_{\sigma}^{p} \mathd \mu_{t}
     \right)^{\frac{1}{p}} <C_{5}   \mathe^{- \sigma c t} . \]
\end{corollary}

\begin{proof}
  Putting $C_{6} = \sup \left( \gamma / \mathring{\gamma} \right)$, we obtain
  \begin{eqnarray*}
    \left( \int_{M_{t}} | h |^{2 r}  g_{\sigma}^{p} \mathd \mu_{t}
    \right)^{\frac{1}{p}} \leqslant \left( \int_{M_{t}} \gamma^{r}
    g_{\sigma}^{p} \mathd \mu_{t} \right)^{\frac{1}{p}} & \leqslant & \left(
    \int_{M_{t}} \left( C_{6}   \mathring{\gamma} \right)^{r} \mathe^{2p
    \sigma c t}  f_{\sigma}^{p} \mathd \mu_{t} \right)^{\frac{1}{p}}\\
    & \leqslant & C_{6}^{r/p}   \mathe^{2 \sigma c t} \left( \int_{M_{t}}
    f_{\sigma + \frac{r}{p}}^{p}   \mathd \mu_{t} \right)^{\frac{1}{p}} .
  \end{eqnarray*}
  With $r/p \leqslant \frac{1}{2} \varepsilon^{2} / \sqrt{p}$ and $\sigma
  +r/p<  \varepsilon^{2}  / \sqrt{p}$, the conclusion follows from Lemma
  \ref{pnorm}.
\end{proof}

From Lemma \ref{ptf}, we have an estimate for the time derivative of
$g_{\sigma}$.
\[ \dt g_{\sigma} \leqslant \Delta g_{\sigma} + \frac{2C_{1}}{\normho} |
   \nabla g_{\sigma} | | \nabla H | - \frac{2 \varepsilon  g_{\sigma}}{C_{1}
   \normho^{2}} | \nabla H |^{2} +2 \sigma | h |^{2} g_{\sigma} . \]
For $k>0$, define $g_{\sigma ,k} = \max \{ g_{\sigma} -k,0 \}$, $A ( k,t ) =
\tmop{supp}  g_{\sigma ,k} = \overline{\{ x \in M_{t}   |  g_{\sigma} ( x ) >k
\}}$. Letting $p \geqslant 8C_{1}^{3} / \varepsilon$, we have
\begin{eqnarray*}
  \frac{\mathd}{\mathd t} \int_{M_{t}} g_{\sigma ,k}^{p} \mathd \mu_{t} &
  \leqslant & p \int_{M_{t}} g_{\sigma ,k}^{p-1} \left[ \Delta g_{\sigma} +
  \tfrac{2C_{1}}{\normho} | \nabla g_{\sigma} | | \nabla H | - \tfrac{2
  \varepsilon g_{\sigma}}{C_{1} \normho^{2}} | \nabla H |^{2} +2 \sigma | h
  |^{2} g_{\sigma} \right] \mathd \mu_{t}\\
  & \leqslant & p \int_{M_{t}}   \left[ - ( p-1 ) g_{\sigma ,k}^{p-2} |
  \nabla g_{\sigma} |^{2} + \tfrac{2C_{1}  g_{\sigma ,k}^{p-1}}{\normho} |
  \nabla g_{\sigma} | | \nabla H | - \tfrac{2 \varepsilon g_{\sigma
  ,k}^{p}}{C_{1} \normho^{2}} | \nabla H |^{2} \right] \mathd \mu_{t}\\
  &  & + 2p \sigma \int_{A ( k,t )} | h |^{2} g_{\sigma}^{p} \mathd \mu_{t}\\
  & \leqslant & - \frac{1}{2} p ( p-1 ) \int_{M_{t}} g_{\sigma ,k}^{p-2} |
  \nabla g_{\sigma} |^{2} \mathd \mu_{t} +2p \sigma \int_{A ( k,t )} | h |^{2}
  g_{\sigma}^{p} \mathd \mu_{t} .
\end{eqnarray*}
Notice that $\frac{1}{2} p ( p-1 ) g_{\sigma ,k}^{p-2} | \nabla g_{\sigma}
|^{2} \geqslant | \nabla g_{\sigma ,k}^{p/2} |^{2}$. Setting $v=g_{\sigma
,k}^{p/2}$, we obtain
\begin{equation}
  \frac{\mathd}{\mathd t} \int_{M_{t}} v^{2} \mathd \mu_{t} + \int_{M_{t}} |
  \nabla v |^{2} \mathd \mu_{t} \leqslant 2p \sigma \int_{A ( k,t )} | h |^{2}
  g_{\sigma}^{p} \mathd \mu_{t} . \label{ineq1}
\end{equation}

The volume of $\tmop{supp}  v$ satisfies $\int_{A ( k,t )} \mathd \mu_{t}
\leqslant \int_{A ( k,t )} g_{\sigma}^{p} k^{-p} \mathd \mu_{t} <C_{5}^{p}
k^{-p}$. It is sufficiently small when $k$ is large. By Theorem 2.1 of
{\cite{MR0365424}}, for the function $u=v^{2 ( n-1 ) / ( n-2 )}$ we have a
Sobolev inequality $\left( \int_{M_{t}} u^{\frac{n}{n-1}} \mathd \mu_{t}
\right)^{\frac{n-1}{n}} \leqslant C_{7} \int_{M_{t}} ( | \nabla u | +u | H | )
\mathd \mu_{t}$, where $C_{7}$ is a positive constant depending only on $n$.
Using H{\"o}lder's inequality, we have
\[ \left( \int_{M_{t}} v^{\frac{2n}{n-2}} \mathd \mu_{t}
   \right)^{\frac{n-2}{n}} \leqslant C_{7} \int_{M_{t}} | \nabla v |^{2}
   \mathd \mu_{t} +C_{7} \left( \int_{A ( k,t )} H^{n} \mathd \mu_{t}
   \right)^{\frac{2}{n}} \left( \int_{M_{t}} v^{\frac{2n}{n-2}} \mathd \mu_{t}
   \right)^{\frac{n-2}{n}} . \]
When $k \geqslant C_{5} ( 2n C_{7} )^{\frac{n}{2p}}$, $p \geqslant n^{2} /
\varepsilon^{4}$ and $\sigma \leqslant \frac{1}{2} \varepsilon^{2} /
\sqrt{p}$, it follows from Corollary \ref{pnorm2} that $\left( \int_{A ( k,t
)} H^{n} \mathd \mu_{t} \right)^{\frac{2}{n}} \leqslant \left( \int_{A ( k,t
)} n^{\frac{n}{2}} | h |^{n} g_{\sigma}^{p} k^{-p} \mathd \mu_{t}
\right)^{\frac{2}{n}} < \frac{1}{2 C_{7}}$. Thus we obtain
\begin{equation}
  \frac{1}{2 C_{7}} \left( \int_{M_{t}} v^{\frac{2n}{n-2}} \mathd \mu_{t}
  \right)^{\frac{n-2}{n}} \leqslant \int_{M_{t}} | \nabla v |^{2} \mathd
  \mu_{t} . \label{vlessdv}
\end{equation}

It follows from (\ref{ineq1}) and (\ref{vlessdv}) that
\begin{equation}
  \frac{\mathd}{\mathd t} \int_{M_{t}} v^{2} \mathd \mu_{t} + \frac{1}{2C_{7}}
  \left( \int_{M_{t}} v^{\frac{2n}{n-2}} \mathd \mu_{t}
  \right)^{\frac{n-2}{n}} \leqslant  2p \sigma \int_{A ( k,t )} | h |^{2}
  g_{\sigma}^{p} \mathd \mu_{t} . \label{dtintv2}
\end{equation}
Let $k \geqslant \sup_{M_{0}}  g_{\sigma}$. Then $v |_{_{t=0}} =0$. For
a fixed time $s$, let $s_{1} \in [ 0,s ]$ be the time when $\int_{M_{t}} v^{2}
\mathd \mu_{t}$ achieves its maximum in $[ 0,s ]$. Integrating (\ref{dtintv2})
over $[ 0,s_{1} ]$ and $[ 0,s ]$ respectively, adding them, we obtain
\begin{eqnarray}
  \underset{t \in [ 0,s ]}{\sup} \int_{M_{t}} v^{2} \mathd \mu_{t} & + &
  \frac{1}{2C_{7}} \int_{0}^{s} \left( \int_{M_{t}} v^{\frac{2n}{n-2}} \mathd
  \mu_{t} \right)^{\frac{n-2}{n}} \mathd t \nonumber\\
  & \leqslant & 4p \sigma \int_{0}^{s} \int_{A ( k,t )} | h |^{2}
  g_{\sigma}^{p} \mathd \mu_{t}   \mathd t.  \label{vineq1}
\end{eqnarray}
Let $\| A ( k ) \|_{s} = \int_{0}^{s} \int_{A ( k,t )} \mathd \mu_{t}   \mathd
t$. For a positive number $r> \frac{n+2}{2}$, let $p \geqslant 4r/
\varepsilon^{4}$, $\sigma \leqslant \frac{1}{2} \varepsilon^{2} / \sqrt{p r}$.
Using H{\"o}lder's inequality and Corollary \ref{pnorm2} , we have
\begin{eqnarray}
  \int_{0}^{s} \int_{A ( k,t )} | h |^{2} g_{\sigma}^{p} \mathd \mu_{t}
  \mathd t & \leqslant & \left( \int_{0}^{s} C_{5}^{p r} \mathe^{-p r  \sigma
  c t} \mathd t \right)^{\frac{1}{r}} \| A ( k ) \|_{s}^{1- \frac{1}{r}}
  \nonumber\\
  & \leqslant & \frac{C_{5}^{p}}{( p r  \sigma  c )^{1/r}} \| A ( k )
  \|_{s}^{1- \frac{1}{r}} .
\end{eqnarray}
For $h>k$, we have $f_{\sigma ,k} >h-k$ on $A ( h,t )$. So
\begin{equation}
  ( h-k )^{p} \| A ( h ) \|_{s} \leqslant \int_{0}^{s} \int_{M_{t}} v^{2}
  \mathd \mu_{t}   \mathd t \leqslant \left( \int_{0}^{s} \int_{M_{t}}
  v^{\frac{2 ( n+2 )}{n}} \mathd \mu_{t}   \mathd t \right)^{\frac{n}{n+2}} \|
  A ( k ) \|_{s}^{\frac{2}{n+2}} .
\end{equation}
We estimate the right hand side of the above inequality as follows.
\begin{eqnarray}
  & &\left( \int_{0}^{s} \int_{M_{t}} v^{\frac{2 ( n+2 )}{n}} \mathd \mu_{t}
  \mathd t \right)^{\frac{n}{n+2}} \nonumber\\
    & \leqslant & \left[ \int_{0}^{s} \left(
  \int_{M_{t}} v^{2} \mathd \mu_{t} \right)^{\frac{2}{n}} \left( \int_{M_{t}}
  v^{\frac{2n}{n-2}} \mathd \mu_{t} \right)^{\frac{n-2}{n}} \mathd t
  \right]^{\frac{n}{n+2}} \nonumber\\
  & \leqslant & \left( \underset{t \in [ 0,s ]}{\sup} \int_{M_{t}} v^{2}
  \mathd \mu_{t} \right)^{\frac{2}{n+2}} \left[ \int_{0}^{s} \left(
  \int_{M_{t}} v^{\frac{2n}{n-2}} \mathd \mu_{t} \right)^{\frac{n-2}{n}}
  \mathd t \right]^{\frac{n}{n+2}}  \label{vineq5}\\
  & \leqslant & C_{8} \cdot \underset{t \in [ 0,s ]}{\sup} \int_{M_{t}} v^{2}
  \mathd \mu_{t} + \frac{C_{8}}{2 C_{7}}   \int_{0}^{s} \left( \int_{M_{t}}
  v^{\frac{2n}{n-2}} \mathd \mu_{t} \right)^{\frac{n-2}{n}} \mathd t.
  \nonumber
\end{eqnarray}
Putting inequalities (\ref{vineq1})-(\ref{vineq5}) together, for $h>k$, we
have
\[ ( h-k )^{p} \| A ( h ) \|_{s} \leqslant   C_{9} \| A ( k ) \|_{s}^{1-
   \frac{1}{r} + \frac{2}{n+2}} , \]
where $C_{9}$ is a positive constant depending on $M_{0}$, $p$, $\sigma$ and
$r$.

By a lemma of {\cite{MR1786735}} (Chapter II, Lemma B.1), there exists a
finite number $k_{1}$, such that $\| A ( k_{1} ) \|_{s} =0$ for all $s$. Hence
we have $g_{\sigma} \leqslant k_{1}$. This completes the proof of Theorem
\ref{sa0h2}.

\section{A gradient estimate}

To compare the mean curvature at different points of $M_{t}$, we need an
estimate for the gradient of mean curvature.

\begin{theorem}
  \label{dH2}For all $\eta \in \left( 0, \frac{1}{n} \right)$, there exists a
  constant $C ( \eta )$ depending on $\eta$ and $M_{0}$, such that
  \[ | \nabla H |^{2} < [ ( \eta  H )^{4} +C ( \eta )^{2} ] \mathe^{- \sigma c
     t} . \]
\end{theorem}

Firstly, we derive an inequality for the time derivative of $| \nabla H |$.

\begin{lemma}
  \label{dtdelH2}There exists a positive constant $B_{1} >1$ depending only on
  $n$, such that
  \[ \dt | \nabla H |^{2} \leqslant \Delta | \nabla H |^{2} +B_{1}   ( H^{2}
     +c ) | \nabla h |^{2} . \]
\end{lemma}

\begin{proof}
  From the evolution equation of $H$, we have
  \begin{eqnarray*}
    \dt \nabla_{i} H & = & \nabla_{i} \left( \dt H \right)\\
    & = & \nabla_{i} ( \Delta H+H ( | h |^{2} +n c ) )\\
    & = & \nabla_{i} \Delta H+ \nabla_{i} H  ( | h |^{2} +n c ) +H \nabla_{i}
    | h |^{2} .
  \end{eqnarray*}
  Lemma \ref{evo} (i) implies $\dt g^{i j} =2 H h^{i j}$. Thus we have
  \begin{eqnarray}
    \dt | \nabla H |^{2} & = & \dt ( g^{i j} \nabla_{i} H  \nabla_{j} H )
    \nonumber\\
    & = & 2 \langle \nabla \Delta H, \nabla H \rangle +2 | \nabla H |^{2}   (
    | h |^{2} +n c )  \label{dtdelHnorm}\\
    &  & +2 H \langle \nabla | h |^{2} , \nabla H \rangle  +2 H h^{i j}
    \nabla_{i} H  \nabla_{j} H. \nonumber
  \end{eqnarray}
  The Laplacian of $| \nabla H |^{2}$ is given by
  \begin{equation}
    \Delta | \nabla H |^{2} =2  \langle \Delta \nabla H, \nabla H \rangle +2
    | \nabla^{2} H |^{2} . \label{lapdelHnorm}
  \end{equation}
  From the Gauss equation, we get
  \begin{equation}
    \nabla \Delta H- \Delta \nabla H= ( 1-n ) c  \nabla H+h^{j k} h_{i j}
    \nabla_{k} H  \mathd x^{i} -H h_{i}^{k} \nabla_{k} H  \mathd x^{i} .
    \label{gaussH}
  \end{equation}
  Combining (\ref{dtdelHnorm}), (\ref{lapdelHnorm}) and (\ref{gaussH}), we
  obtain the evolution equation of $| \nabla H |^{2}$.
  \begin{eqnarray}
    \dt | \nabla H |^{2} & = & \Delta | \nabla H |^{2} -2  | \nabla^{2} H
    |^{2} +2 | \nabla H |^{2}   ( | h |^{2} +c  ) \nonumber\\
    &  & +2 H \langle \nabla | h |^{2} , \nabla H \rangle +2 h^{i j}
    h_{i}^{k}   \nabla_{j} H  \nabla_{k} H.
  \end{eqnarray}

  Using the Cauchy-Schwarz inequality, we have $H \langle \nabla | h |^{2} ,
  \nabla H \rangle \leqslant 2 | H | | h | | \nabla h | | \nabla H |$ and
  $h^{i j} h_{i}^{k} \nabla_{j} H \nabla_{k} H \leqslant | h |^{2} | \nabla H
  |^{2}$. Using $| h |^{2} < \gamma$ and $| \nabla H |^{2} \leqslant
  \frac{n+2}{3} | \nabla h |^{2}$, we obtain the conclusion.
\end{proof}

Secondly, we need the following estimates.

\begin{lemma}
  \label{dtH4}Along the mean curvature flow, we have
  \begin{enumerateroman}
    \item $\dt H^{4} \geqslant \Delta H^{4} -7 n H^{2} | \nabla h |^{2} +
    \frac{4}{n} H^{6}$,

    \item $\dt \normho^{2} \leqslant \Delta \normho^{2} - \frac{8}{9} | \nabla
    h |^{2} +H^{2} \normho^{2}$,

    \item $\dt \left( H^{2} \normho^{2} \right) \leqslant \Delta \left( H^{2}
    \normho^{2} \right) - \frac{7}{9}  H^{2} | \nabla h |^{2} +B_{2} | \nabla
    h |^{2} +4n H^{2} ( H^{2} +c ) \normho^{2}$, where $B_{2} >2 c$ is a
    positive constant depending on $M_{0}$.
  \end{enumerateroman}
\end{lemma}

\begin{proof}

  (i) We derive that
  \[ \dt H^{4} = \Delta H^{4} -12 H^{2} | \nabla H |^{2} +4 H^{4} ( | h |^{2}
     +n c ) . \]
  This together with $| h |^{2} \geqslant \frac{1}{n} H^{2}$ and $| \nabla H
  |^{2} \leqslant \frac{n+2}{3} | \nabla h |^{2}$ implies inequality (i).

  (ii) The evolution equation of $\normho^{2}$ is given by
  \[ \dt \normho^{2} = \Delta \normho^{2} -2  | \nabla h |^{2} + \frac{2}{n} |
     \nabla H |^{2} +2 \normho^{2} ( | h |^{2} -n c ) . \]
  Using $\frac{1}{n} | \nabla H |^{2} \leqslant \frac{n+2}{3n} | \nabla
  h |^{2} \leqslant \frac{5}{9} | \nabla h |^{2}$ and $| h |^{2} -n c< \gamma
  -n c \leqslant \frac{1}{2} H^{2}$, we get inequality (ii).

  (iii) It follows from the evolution equations that
  \begin{eqnarray*}
    \dt \left( H^{2} \normho^{2} \right) & = & \Delta \left( H^{2} \normho^{2}
    \right) +4 H^{2} | h |^{2} \normho^{2} -2 H^{2} \left( | \nabla h |^{2} -
    \frac{1}{n} | \nabla H |^{2} \right)\\
    &  & -2 \normho^{2} | \nabla H |^{2} -4 H \left\langle \nabla H, \nabla
    \normho^{2} \right\rangle .
  \end{eqnarray*}
  We use $\frac{1}{n} | \nabla H |^{2} \leqslant \frac{5}{9} | \nabla h |^{2}$
  again. From $| h |^{2} < \gamma$, we have $4H^{2} | h |^{2} \normho^{2}
  \leqslant 4n H^{2} ( H^{2} +c ) \normho^{2}$. From the formula $\nabla_{i}
  \normho^{2} =2 \mathring{h}^{j k}   \nabla_{i} h_{j k}$ and Young's
  inequality, we get
  \begin{eqnarray*}
    -4 H \left\langle \nabla H, \nabla \normho^{2} \right\rangle & \leqslant &
    8  | H |   | \nabla H |   \normho   | \nabla h |\\
    & \leqslant & 8 \sqrt{\tfrac{n+2}{3} C_{0}}   | H | ( H^{2} + c
    )^{\frac{1- \sigma}{2}} | \nabla h |^{2}  \\
    & \leqslant & \left( B_{2} + \frac{1}{9}  H^{2} \right) | \nabla h |^{2}
    .
  \end{eqnarray*}
  This proves inequality (iii).
\end{proof}

\noindent\textit{Proof of Theorem \ref{dH2}.}
Define the following scalar
\[ f= \left( | \nabla H |^{2} +9B_{1} B_{2}   \normho^{2} +7B_{1}  H^{2}
   \normho^{2} \right) \mathe^{\sigma c t} - ( \eta  H )^{4} , \hspace{1em}
   \eta \in \left( 0, \frac{1}{n} \right) . \]
From Lemma \ref{dtdelH2} and \ref{dtH4}, we obtain
\begin{eqnarray*}
  \left( \dt - \Delta \right) f & \leqslant & \left[ B_{1}   ( H^{2} +c ) |
  \nabla h |^{2} +9B_{1} B_{2} \left( - \frac{8}{9} | \nabla h |^{2} +H^{2}
  \normho^{2} \right) \right.\\
  &  & +7B_{1} \left( - \frac{7}{9}  H^{2} | \nabla h |^{2} +B_{2} | \nabla h
  |^{2} +4n H^{2} ( H^{2} +c ) \normho^{2} \right)\\
  &  & \left. + \sigma c \left( | \nabla H |^{2} +9B_{1} B_{2}   \normho^{2}
  +7B_{1}  H^{2} \normho^{2} \right) \right] \mathe^{\sigma c t}\\
  &  & - \eta^{4} \left( -7 n H^{2} | \nabla h |^{2} + \frac{4}{n} H^{6}
  \right)\\
  & \leqslant & \left( - \frac{40}{9} B_{1} \mathe^{\sigma c t} +7n \eta^{4}
  \right) H^{2} | \nabla h |^{2} + [ ( B_{1} c-B_{1} B_{2} ) | \nabla h |^{2}
  + \sigma c | \nabla H |^{2} ] \mathe^{\sigma c t}\\
  &  & +B_{3} ( H^{2} +c )^{2} \normho^{2} \mathe^{\sigma c t} - \frac{4
  \eta^{4}}{n} H^{6} .
\end{eqnarray*}
By Theorem \ref{sa0h2}, we get
\[ \left( \dt - \Delta \right) f \leqslant \left[ C_{0} B_{3} ( H^{2} +c )^{3-
   \sigma} - \frac{4 \eta^{4}}{n} H^{6} \right] \mathe^{- \sigma c t} . \]
We consider the expression in the bracket of the right hand side of the above
inequality, which it is a function of $H$. Since the coefficient of the
highest-degree term is negative, the supremum $C_{2} ( \eta )$ of this
function is finite. Then we have $\dt f< \Delta f+C_{2} ( \eta ) \mathe^{-
\sigma c t}$. It follows from the maximum principle that $f$ is bounded. This
completes the proof of Theorem \ref{dH2}.
\hfill\qedsymbol\\

\section{Convergence Under Sharp Pinching Condition}

In order to estimate the diameter of $M_{t}$, we need a lower bound for the
Ricci curvature.

\begin{lemma}
  \label{ric}Suppose that $M$ is an n-dimensional $(n \geqslant 3 )$
  hypersurface in $\mathbb{F}^{n+1} (c)$ satisfying $|
  h |^{2} < \gamma - \varepsilon   \omega$. Then there exists a positive
  constant $B_{4}$ depending only on $n$, such that for any unit vector $X$ in
  the tangent space, the Ricci curvature of $M$ satisfies
  \[ \tmop{Ric} ( X ) \geqslant B_{4}   \varepsilon   ( H^{2} +c ) . \]
\end{lemma}

\begin{proof}
  Using Proposition 2 of {\cite{MR1458750}} and Lemma \ref{app} (vi), we have
  \begin{eqnarray*}
    \tmop{Ric} ( X ) & \geqslant & \frac{n-1}{n} \left( n c+ \frac{2}{n} H^{2}
    - | h |^{2} - \frac{n-2}{\sqrt{n ( n-1 )}} | H | \normho \right)\\
    & > & \frac{n-1}{n} \left( n c+ \frac{2}{n} H^{2} - ( \gamma -
    \varepsilon   \omega ) - \frac{n-2}{\sqrt{n ( n-1 )}} | H |
    \sqrt{\mathring{\gamma}} \right)\\
    & \geqslant & \frac{n-1}{n}   \varepsilon   \omega\\
    & > & B_{4}   \varepsilon   ( H^{2} +c ) .
  \end{eqnarray*}

\end{proof}

We also need the well-known Myers theorem.

\begin{theorem}
  [\textbf{Myers Theorem}] Let $\Gamma$ be a geodesic of length at least
  $\pi / \sqrt{k}$ in $M$. If the Ricci curvature satisfies $\tmop{Ric} ( X )
  \geqslant ( n-1 ) k$, for each unit vector $X \in T_{x}
  M$, at any point $x \in \Gamma$, then $\Gamma$ has conjugate points.
\end{theorem}

Now we show that $M_{t}$ converges to a round point or a totally geodesic
sphere.

\begin{theorem}
  \label{Tfin}If $T$ is finite, then $F_{t}$ converges to a round point as $t
  \rightarrow T$.
\end{theorem}

\begin{proof}
  Let $| H |_{\min} = \min_{M_{t}} | H |$, $| H |_{\max} = \max_{M_{t}} | H
  |$. By Theorem \ref{dH2}, for any $\eta \in \left( 0, \frac{1}{n} \right)$,
  there exists $C ( \eta ) >1$ such that $| \nabla H | < ( \eta  H )^{2} +C (
  \eta )$. From Theorem 7.1 of {\cite{MR837523}}, $| H |_{\max}$ becomes
  unbounded as $t \rightarrow T$. So, there exists a time $\tau$ depending on
  $\eta$, such that $| H |_{\max}^{2} >C ( \eta ) / \eta^{2}$ on $M_{\tau}$.
  Then we have $| \nabla H | <2  \eta^{2} | H |^{2}_{\max}$ on $M_{\tau}$.

  Let $x$ be a point in $M_{\tau}$ where $| H |$ achieves its maximum. Then
  along all geodesics of length $l= ( 2  \eta   | H |_{\max} )^{-1}$ starting
  from $x$, we have $| H | > | H |_{\max} -2\eta^{2}|H|^{2}_{\max} \cdot l= ( 1- \eta )
  | H |_{\max}$. With $\eta$ small enough, Lemma \ref{ric} implies $\tmop{Ric}
  >B_{4} \varepsilon ( 1- \eta )^{2} | H |^{2}_{\max} > ( n-1 ) \pi^{2}
  /l^{2}$ on these geodesics. Then by Myers' theorem, these geodesics can
  reach any point of $M_{\tau}$.

  Then we have $| H | > ( 1- \eta ) | H |_{\max} > \frac{1}{2 \eta}$ on
  $M_{\tau}$. Thus we can assume $H> \frac{1}{2 \eta}$ on $M_{\tau}$ without
  loss of generality. Let $\eta$ be sufficiently small. From Theorem
  \ref{sa0h2}, at any point in $M_{\tau}$, the principal curvatures
  $\lambda_{i} \geqslant \frac{H}{n} - \normho >0$, $1 \leqslant i \leqslant n$. Hence $M_{\tau}$ is a
  convex hypersurface. By Theorem 1.1 of {\cite{MR837523}}, $F_{t}$ will
  converge to a round point (see \cite{Zhu} Chapter 11 for the details).
\end{proof}

\begin{theorem}
  \label{Tinf}If $T= \infty$, then $F_{t}$ converges to a totally geodesic
  hypersurface as $t \rightarrow \infty$.
\end{theorem}

\begin{proof}
  Firstly we prove that $| H |_{\max}$ must remains bounded for all $t \in [
  0, \infty )$. If not, we see that $F_{t}$ will converge to a round point in
  finite time from the proof of Theorem \ref{Tfin}.

  Next we prove $| H |_{\min} =0$ for all $t \in [ 0, \infty )$. Suppose $| H
  |_{\min} >0$ on $M_{\theta}$. From the evolution equation of $H$, we have
  $\dt | H | > \Delta | H | + \frac{1}{n} | H |^{3}$ for $t \geqslant \theta$.
  By the maximum principle, we get that $| H |$ will tend to infinity in
  finite time. This leads to a contradiction.

  By Lemma \ref{ric} and Myers' theorem, $\tmop{diam}  M_{t}$ is uniformly
  bounded. Applying Theorem \ref{dH2}, we have $| H |
  <C \mathe^{- \frac{\sigma c t}{2}}$. Therefore, from
  Theorem \ref{sa0h2} we obtain $| h |^{2} = \normho^{2} + \frac{H^{2}}{n}
  \leqslant C^{2} \mathe^{- \sigma c t}$.

  By Lemma 7.2 of {\cite{MR837523}}, $| \nabla^{m} h |$ is bounded for all $m
  \in \mathbb{N}$ and $t \in [ 0,+ \infty )$. Then $M_{t}$ converges to a
  smooth limit hypersurface $M_{\infty}$. Since $| h | \rightarrow 0$ as $t
  \rightarrow \infty$, $M_{\infty}$ is totally geodesic.
\end{proof}

\noindent\textit{Proof of Theorem \ref{theo1}.}
Combining Theorem \ref{Tfin} and \ref{Tinf}, we complete the
proof of Theorem \ref{theo1}.
\hfill\qedsymbol\\

\section{Convergence Under Weakly Pinching Condition}

Now we are in the position to prove Theorem \ref{theo2}. Assume that $M_{0}$
is a closed hypersurface immersed in $\mathbb{S}^{n+1} \left( 1/ \sqrt{c}
\right)$, and $M_{0}$ satisfies $| h |^{2} \leqslant \gamma$.

\noindent\textit{Proof of Theorem \ref{theo2}.}
  Recall the proof of Theorem \ref{pinch}. Letting $\varepsilon =0$, we get
  \begin{eqnarray}
    \left( \dt - \Delta \right) ( | h |^{2} - \gamma ) & \leqslant & \left[ -
    \frac{6}{n+2} +2 ( 2H^{2} \gamma'' + \gamma' ) \right] | \nabla H |^{2}
    \nonumber\\
    &  & +2 ( | h |^{2} - \gamma )^{2} +2 ( | h |^{2} - \gamma ) ( 2 \gamma
    -H^{2} \gamma' -n c )  \label{weakpinch}\\
    &  & +4 c H^{2} +2  \gamma^{2} -2 n c  \gamma -2 ( \gamma +n c ) H^{2}
    \gamma' . \nonumber
  \end{eqnarray}
  By Lemma \ref{app} (i) and (ii), we get $\left( \dt - \Delta \right) ( | h |^{2}
  - \gamma ) \leqslant 2 ( | h |^{2} - \gamma ) ( | h |^{2} + \gamma -H^{2}
  \gamma' -n c )$. Then using the strong maximum principle, we have either $|
  h |^{2} < \gamma$ at some time $t_{0} \in ( 0,T )$, or $| h |^{2} \equiv \gamma$ for all $t
  \in [ 0,T )$.

  If $| h |^{2} < \gamma$ at some $t_{0}$, it reduces to the case of Theorem
  \ref{theo1}.

  If $| h |^{2} \equiv \gamma$ for $t \in [ 0,T )$, we have $\left[ -
  \frac{6}{n+2} +2 ( 2H^{2} \gamma'' + \gamma' ) \right] | \nabla H |^{2}
  \equiv 0$ and $ 4 c H^{2} +2  \gamma^{2} -2 n c  \gamma -2 ( \gamma +n c )
  H^{2} \gamma'   \equiv  0$ for all $t$. Thus we obtain $\nabla H=0$ and
  $H^{2} \geqslant x_{0}$. By Theorem B, $M_{t}$ is the isoparametric
  hypersurface
  \[ \mathbb{S}^{n-1} \left( \frac{1}{\sqrt{c+ \lambda^{2}}} \right) \times
     \mathbb{S}^{1} \left( \frac{\lambda}{\sqrt{c^{2} +c  \lambda^{2}}}
     \right) , \]
  where $\lambda = \frac{| H | + \sqrt{ H ^{2} +4 ( n-1 )c}}{2 ( n-1 )}
    > \sqrt{\frac{c}{n-1}}$.

  We see that $\lambda$ is the ($n-1$)-multiple principal curvature and $-
  \frac{c}{\lambda}$ is the other one. Thus we have $| H | = ( n-1 ) \lambda -
  \frac{c}{\lambda}$ and $| h |^{2} = ( n-1 ) \lambda^{2} +
  \frac{c^{2}}{\lambda^{2}}$. Substituting these equalities into the evolution equation of
  $H$, we get
  \[ \frac{\mathd}{\mathd t} \left( ( n-1 ) \lambda - \frac{c}{\lambda}
     \right) = \left( ( n-1 ) \lambda - \frac{c}{\lambda} \right) \left( ( n-1
     ) \lambda^{2} + \frac{c^{2}}{\lambda^{2}} +n c \right) . \]
  Let $r_{1} = \frac{1}{\sqrt{c+ \lambda^{2}}}$, $r_{2} =
  \frac{\lambda}{\sqrt{c^{2} +c  \lambda^{2}}}$. The above equation implies
  \[ \frac{\mathd}{\mathd t} r_{1}^{2} =2-2n+2n c r_{1}^{2} . \]
  Solving the ODE above, we obtain
  \[ r_{1}^{2} = \frac{n-1}{n c} ( 1-d \cdot \mathe^{2n c t} ) , \]
  where $d \in ( 0,1 )$ is a constant of integration.

  It's seen from the solution of $r_{1}$ that the maximal existence time $T=- \frac{\log
  d}{2n c}$. Hence we obtain $r_{1}^{2} = \frac{n-1}{n c} ( 1- \mathe^{2n
  c ( t-T )} )$. We see that $M_{t}$ converges to a great circle
  $\mathbb{S}^{1} \left( 1/ \sqrt{c} \right)$ as $t \rightarrow T$. This
  completes the proof of Theorem \ref{theo2}.
\hfill\qedsymbol\\

Motivated by Theorem B and Theorem \ref{theo2}, we propose the
following conjecture for the mean curvature flow in a sphere.

\begin{conjecture}
  Let $F_{0} :M^{n} \rightarrow \mathbb{S}^{n+1} \left( 1/ \sqrt{c} \right)$
  be an n-dimensional closed hypersurface immersed in a
  sphere. If $F_{0}$ satisfies $| h |^{2} \leqslant \alpha ( n,H,c )$, then
  the mean curvature flow with initial value $F_{0}$ has a unique smooth
  solution $F: M \times [ 0,T ) \rightarrow \mathbb{S}^{n+1} \left( 1/
  \sqrt{c} \right)$, and either
  \begin{enumerateroman}
    \item $T$ is finite, and $F_{t}$ converges to a round point as $t
    \rightarrow T$,

    \item $T= \infty$, and $F_{t}$ converges to a totally geodesic sphere as
    $t \rightarrow \infty$, or

    \item $T$ is finite, $M_{t}$ is congruent to $\mathbb{S}^{n-1} (
    r_{1} ( t ) ) \times \mathbb{S}^{1} ( r_{2} ( t ) )$, where $r_{1} ( t
    )^{2} +r_{2} ( t )^{2} =1/c$, $r_{1} ( t )^{2} = \frac{n-1}{n c} ( 1-
    \mathe^{2n c ( t-T )} )$, and $F_{t}$ converges to a great circle as $t
    \rightarrow T$,

    \item $T= \infty$, and $M_{t}$ is congruent to one of the minimal
    hypersurfaces $\mathbb{S}^{k} \left( \sqrt{\frac{k}{n c}} \right) \times
    \mathbb{S}^{n-k} \left( \sqrt{\frac{n-k}{n c}} \right)$, $k=1, \cdots
    ,n-1$.
  \end{enumerateroman}
\end{conjecture}

\end{document}